\documentclass[letterpaper,12pt]{amsart}
\usepackage{amsmath,amssymb,amsxtra,amsthm, amstext, amscd,amsfonts,fancyhdr, hyperref}
\newcommand{\bP}{{\mathbb{P}}}
\newcommand{\bQ}{{\mathbb{Q}}}
\newcommand{\bR}{{\mathbb{R}}}

\newcommand{\bZ}{{\mathbb{Z}}}

\newcommand{\Bx}{{\mathbf{x}}}

  \newcommand{\G}{{\mathcal{G}}}

  \newcommand{\N}{{\mathcal{N}}}

  \newcommand{\R}{{\mathcal{R}}}
\renewcommand{\S}{{\mathcal{S}}}

\newcommand{\ord}{\operatorname{ord}}

\newcommand{\GL}{\operatorname{GL}}

\newcommand{\Aut}{\operatorname{Aut}}
\newcommand{\DD}{\mathbf{D}}

\newcommand{\CC}{\mathbf{C}}

\newcommand{\Rep}{\operatorname{Rep}}

\newcommand{\ep}{\varepsilon}

\newcommand{\lcm}{\operatorname{lcm}}

\newcommand{\upchi}{{\raise.35ex\hbox{$\chi$}}}

\makeatletter
\@namedef{subjclassname@2010}{%
  \textup{2010} Mathematics Subject Classification}
\makeatother

\newtheorem{theorem}{Theorem}[section]
\newtheorem{corollary}[theorem]{Corollary}

\newtheorem{lemma}[theorem]{Lemma}

\theoremstyle{definition}

\numberwithin{equation}{section}

\begin{document}

\title{On the representation of integers by binary forms}

\author{C.~L.~Stewart}
\address{Department of Pure Mathematics \\
University of Waterloo \\
Waterloo, ON\\  N2L 3G1 \\
Canada}
\email{cstewart@uwaterloo.ca}

\author{Stanley Yao Xiao}
\address{Mathematical Institute\\
University of Oxford \\
Oxford \\  OX2 6GG \\
United Kingdom}
\email{stanley.xiao@maths.ox.ac.uk}

\indent
\thanks{The research of the first author was supported in part by the Canada Research Chairs Program and by Grant A3528 from the Natural Sciences and Engineering Research Council of Canada.}

\subjclass[2010]{Primary 11D45, Secondary 11D59, 11E76}%
\keywords{binary forms, determinant method}%
\date{\today}


\begin{abstract} Let $F$ be a binary form with integer coefficients, non-zero discriminant and degree $d$ with $d$ at least $3$. Let $R_F(Z)$ denote the number of integers of absolute value at most $Z$ which are represented by $F$. We prove that there is a positive number $C_F$ such that $R_F(Z)$ is asymptotic to $C_F Z^{\frac{2}{d}}$. 
\end{abstract}

\maketitle

\section{Introduction}
\label{S1}

Let $F$ be a binary form with integer coefficients, non-zero discriminant $\Delta(F)$ and degree $d$ with $d \geq 2$. For any positive number $Z$ let $\R_F(Z)$ denote the set of non-zero integers $h$ with $|h| \leq Z$ for which there exist integers $x$ and $y$ with $F(x,y) = h$. Denote the cardinality of a set  $\S$ by $|\S|$ and put $R_F(Z) = |\R_F(Z)|$. There is an extensive literature, going back to the foundational work of Fermat, Lagrange, Legendre and Gauss \cite{Gau}, concerning the set $\R_F(Z)$ and the growth of $R_F(Z)$ when $F$ is a binary quadratic form; see \cite{BV}, \cite{Bue} and \cite{Car} for more recent treatments of these topics. For forms of higher degree much less is known. In 1938 Erd\H{o}s and Mahler \cite{EM} proved that if $F$ is irreducible over $\bQ$ and $d$ is at least 3 then there exist positive numbers $c_1$ and $c_2$, which depend on $F$, such that 
\[R_F(Z) > c_1 Z^{\frac{2}{d}}\]
for $Z > c_2$. \\ 

Put 
\begin{equation} \label{AF defn} A_F = \mu (\{(x,y) \in \bR^2 : |F(x,y)| \leq 1\}) \end{equation}
where $\mu$ denotes the area of a set in $\bR^2$. In 1967 Hooley \cite{Hoo1} determined the asymptotic growth rate of $R_F(Z)$ when $F$ is an irreducible binary cubic form with discriminant which is not a square. He proved that

\begin{equation} \label{Hooley result 1} R_F(Z) = A_F Z^{\frac{2}{3}} + O \left(Z^{\frac{2}{3}} (\log \log Z)^{-\frac{1}{600}}\right). \end{equation}
In 2000 Hooley \cite{Hoo2} treated the case when the discriminant is a perfect square. Suppose that
\[F(x,y) = b_3 x^3 + b_2 x^2 y + b_1 xy^2 + b_0 y^3.\]
The Hessian covariant of $F$ is
\[q_F(x) = Ax^2 + Bx + C,\]
where
\[A = b_2^2 - 3 b_3 b_1, \text{ } B = b_2 b_1 - 9 b_3 b_0 \text{ and } C = b_1^2 - 3b_2 b_0.\]
Put
\begin{equation} \label{Hooley m} m = \frac{\sqrt{\Delta(F)}}{\gcd(A,B,C)}.\end{equation}
Hooley proved that if $F$ is an irreducible cubic with $b_1$ and $b_2$ divisible by 3 and $\Delta(F)$ a square then there is a positive number $\gamma$ such that 
\begin{equation} \label{Hooley result 2} R_F(Z) = \left(1 - \frac{2}{3m}\right) A_F Z^{\frac{2}{3}} + O \left(Z^{\frac{2}{3}} (\log Z)^{-\gamma}\right). \end{equation}

We remark that if $F$ is a binary cubic form then
\[\lvert \Delta(F) \rvert^{\frac{1}{6}} A_F = \begin{cases} \dfrac{3 \Gamma^2(1/3)}{\Gamma(2/3)} & \text{if } \Delta(F) > 0, \\ \\ \dfrac{\sqrt{3} \Gamma^2(1/3)}{\Gamma(2/3)} & \text{if } \Delta(F) < 0, \end{cases}\]
where $\Gamma(s)$ denotes the gamma function. In \cite{BE} Bean gives a simple representation for $A_F$ when $F$ is a binary quartic form. \\

Hooley \cite{Hoo1-2} also studied quartic forms of the shape
\[F(x,y) = ax^4 + 2bx^2 y^2 + cy^4.\]
Let $\ep > 0$. He proved that if $a/c$ is not the fourth power of a rational number then
\begin{equation} \label{Hooley result 3} R_F(Z) = \frac{A_F}{4} Z^{\frac{1}{2}} + O_{F,\ep} \left(Z^{\frac{18}{37} + \ep}\right). \end{equation}
Further if $a/c = A^4/C^4$ with $A$ and $C$ coprime positive integers then
\begin{equation} \label{Hooley result 4} R_F(Z) = \frac{A_F}{4} \left(1 - \frac{1}{2AC}\right) Z^{\frac{1}{2}} + O_{F,\ep} \left(Z^{\frac{18}{37} + \ep}\right).\end{equation}

In addition to these results, when $d>2$ and $F$ is the product of $d$ linear forms with integer coefficients Hooley \cite{Hooley1}, \cite{Hooley} proved that there is a positive number $C_F$ such that for each positive number $\ep$
\begin{equation} \label{rep by Hooley} R_F(Z) = C_F Z^{\frac{2}{d}} + O_{F,\ep} \left(Z^{\eta_d + \ep}\right),\end{equation}
where $\eta_3$ is ${\frac{5}{9}}$ and $\eta_d$ is ${\frac{2}{d}}-{\frac{d-2}{d^2(d-1)}}$ if $d>3.$

Further, Browning \cite{B2}, Greaves \cite{Gre}, Heath-Brown \cite{HB3}, Hooley \cite{Hoo1-1}, \cite{Hool1}, \cite{Hool2}, Skinner and Wooley \cite{SWoo} and Wooley \cite{Woo} have obtained asymptotic estimates for $R_F(Z)$ when $F$ is of the form $x^d + y^d$ with $d \geq 3$. Furthermore, Bennett, Dummigan and Wooley \cite{BDW} have obtained an asymptotic estimate for $R_F(Z)$ when $F(x,y) = ax^d + by^d$ with $d \geq 3$ and $a$ and $b$ non-zero integers for which $a/b$ is not the $d$-th power of a rational number. \\

For each binary form $F$ with integer coefficients, non-zero discriminant and degree $d$ with $d\geq 3$ we define $\beta_F$ in the following way. If $F$ has a linear factor in $\bR [x]$ we put

\begin{equation} \label{Beta D} \beta_F = \begin{cases} \dfrac{12}{19} & \text{if } d = 3 \text{ and } F \text{ is irreducible} \\ \\
\dfrac{1}{2} & \text{if } d = 3 \text{ and } F \text{ is reducible} \\ \\
\dfrac{3}{(d-2)\sqrt{d} + 3} & \text{if } 4 \leq d \leq 8 \\ \\ 
\dfrac{1}{d-1} & \text{if } d \geq 9. \end{cases}
\end{equation}
If $F$ does not have a linear factor over $\bR$ then $d$ is even and we put

\begin{equation} \label{Beta DD} \beta_F = \begin{cases} \dfrac{3}{8} & \text{if } d = 4 \\ \\
\dfrac{1}{2\sqrt{6} } & \text{if } d=6\\ \\ 
\dfrac{1}{d-1} & \text{if } d \geq 8. \end{cases}
\end{equation}

We remark that $\beta_F$ only depends on the \emph{real splitting type} of $F$, rather than any particulars of the coefficients of $F$. \\

We shall prove the following result.

\begin{theorem} \label{Rep by BF MT} Let $F$ be a binary form with integer coefficients, non-zero discriminant and degree $d$ with $d\geq 3$. Let $\ep > 0$. There exists a positive number $C_F$ such that
\begin{equation} \label{rep by BF ME} R_F(Z) = C_F Z^{\frac{2}{d}} + O_{F,\ep} \left(Z^{\beta_F + \ep}\right),\end{equation}
where $\beta_F$ is given by (\ref{Beta D}) and (\ref{Beta DD}).
\end{theorem}

Our proof of Theorem \ref{Rep by BF MT} depends on some results of Salberger in \cite{S2} and \cite{S3}, which are based on a refinement of Heath-Brown's $p$-adic determinant method in \cite{HB1}, an argument of Heath-Brown \cite{HB1} and a classical result of Mahler \cite{Mah}.\\

Let $A$ be an element of $\GL_2(\bQ)$ with
\[A = \begin{pmatrix} a_1 & a_2 \\ a_3 & a_4 \end{pmatrix}.\]
Put $F_A(x,y) = F(a_1 x + a_2 y, a_3 x + a_4 y)$. We say that $A$ fixes $F$ if $F_A = F$. The set of $A$ in $\GL_2(\bQ)$ which fix $F$ is the automorphism group of $F$ and we shall denote it by $\Aut F$. Let $G_1$ and $G_2$ be subgroups of $\GL_2(\bQ)$. We say that they are equivalent under conjugation if there is an element $T$ in $\GL_2(\bQ)$ such that $G_1 = T G_2 T^{-1}.$ \\

The positive number $C_F$ in (\ref{rep by BF ME}) is a rational multiple of $A_F$ and the rational multiple depends on $\Aut F$. There are $10$ equivalence classes of finite subgroups of $\GL_2(\bQ)$ under $\GL_2(\bQ)$-conjugation to which $\Aut F$ might belong and we give a representative of each equivalence class together with its generators in Table 1. \\
\begin{center} \label{Table 1}
\begin{tabular}
{ |p{1.5cm}|p{4.5cm}|p{1.5cm}|p{4.5cm}|  }
 \hline
 \multicolumn{4}{|c|}{Table 1} \\
 \hline
 Group & Generators & Group & Generators \\
 \hline
 & & & \\
 $\CC_1$   & $\begin{pmatrix} 1 & 0 \\ 0 & 1 \end{pmatrix}$    & $\DD_1$ &   $\begin{pmatrix} 0 & 1 \\ 1 & 0 \end{pmatrix}$ \\ & & & \\
 $\CC_2$ &   $\begin{pmatrix} -1 & 0 \\ 0 & -1 \end{pmatrix}$  & $\DD_2$ & $\begin{pmatrix} 0 & 1 \\ 1 & 0 \end{pmatrix}, \begin{pmatrix} -1 & 0 \\ 0 & -1 \end{pmatrix}$ \\ & & & \\
 $\CC_3$ & $\begin{pmatrix} 0 & 1 \\ -1 & -1 \end{pmatrix}$ & $\DD_3$ & $\begin{pmatrix} 0 & 1 \\ 1 & 0 \end{pmatrix}, \begin{pmatrix} 0 & 1 \\ -1 & -1 \end{pmatrix}$\\ & & & \\ 
 $\CC_4$    & $\begin{pmatrix} 0 & 1 \\ -1 & 0 \end{pmatrix}$ & $\DD_4$ &  $\begin{pmatrix} 0 & 1 \\ 1 & 0 \end{pmatrix}, \begin{pmatrix} 0 & 1 \\ -1 & 0 \end{pmatrix}$ \\ & & & \\
 $\CC_6$ &  $\begin{pmatrix} 0 & - 1 \\ 1 & 1 \end{pmatrix}$  & $\DD_6$ & $\begin{pmatrix} 0 & 1 \\ 1 & 0 \end{pmatrix}, \begin{pmatrix} 0 & 1 \\ -1 & 1 \end{pmatrix}$ \\ & & & \\
  \hline
\end{tabular}
\end{center}

Since the matrix $-I = \begin{pmatrix} -1 & 0 \\ 0 & -1 \end{pmatrix}$ is in $\Aut F$ if and only if the degree of $F$ is even, we see from an examination of Table 1 that if the degree of $F$ is odd then $\Aut F$ is equivalent to one of $\CC_1, \CC_3, \DD_1$ and $\DD_3$ and if the degree of $F$ is even then $\Aut F$ is equivalent to one of $\CC_2, \CC_4, \CC_6, \DD_2, \DD_4$ and $\DD_6$. \\

Note that the table has fewer entries than Table 1 of \cite{ST1} which gives representatives for the equivalence classes of finite subgroups of $\GL_2(\bZ)$ under $\GL_2(\bZ)$-conjugation. This is because for $i = 1,2,3$ the groups $\DD_i$ and $\DD_i^\ast$ are equivalent under conjugation in $\GL_2(\bQ)$ but are not equivalent under conjugation in $\GL_2(\bZ)$. Further every finite subgroup of $\GL_2(\bQ)$ is conjugate to a finite subgroup of $\GL_2(\bZ)$, see \cite{New}. \\

Let $\Lambda$ be the sublattice of $\bZ^2$ consisting of $(u,v)$ in $\bZ^2$ for which $\displaystyle A \binom{u}{v}$ is in $\bZ^2$ for all $A$ in $\Aut F$, and put
\begin{equation} \label{lattice det} m = d (\Lambda), \end{equation}
where $d(\Lambda)$ denotes the determinant of $\Lambda$.
Note that $m = 1$ when $\Aut F$ is equal to either $\CC_1$ or $\CC_2$. Observe that since $\CC_1$ and $\CC_2$ contain only diagonal matrices, their conjugacy classes over $\GL_2(\bQ)$ consist only of themselves.   \\

When $\Aut F$ is conjugate to $\DD_3$ it has three subgroups $G_1, G_2$ and $G_3$ of order $2$ with generators $A_1, A_2$ and $ A_3$ respectively, and one, $G_4$ say, of order $3$ with generator $A_4$. Let $\Lambda_i = \Lambda(A_i)$ be the sublattice of $\bZ^2$ consisting of $(u,v)$ in $\bZ^2$ for which $A_i \displaystyle \binom{u}{v}$ is in $\bZ^2$ and put
\begin{equation} \label{D3 lattice det} m_i = d(\Lambda_i)\end{equation}
for $i=1,2,3,4$. We remark that $m_4$ is well defined since, by (\ref{C3 C6}), $\Lambda_4$ does not depend on the choice of generator $A_4$. \\

When $\Aut F$ is conjugate to $\DD_4$ there are three subgroups $G_1, G_2$ and $G_3$ of order $2$ of $\Aut F / \{\pm I\}$. Let $\Lambda_i$ be the sublattice of $\bZ^2$ consisting of $(u,v)$ in $\bZ^2$ for which $A \displaystyle \binom{u}{v}$ is in $\bZ^2$ for $A$ in a generator of $G_i$ and put
\begin{equation} \label{D4 lattice det} m_i = d(\Lambda_i)\end{equation}
for $i = 1,2,3$. \\

Finally when $\Aut F$ is conjugate to $\DD_6$ there are three subgroups $G_1, G_2$ and $G_3$ of order $2$ and one, $G_4$ say, of order $3$ in $\Aut F / \{\pm I\}$. Let $A_i$ be in a generator of $G_i$ for $i = 1,2,3,4$. Let $\Lambda_i = \Lambda(A_i)$ be the sublattice of $\bZ^2$ consisting of $(u,v)$ in $\bZ^2$ for which $\displaystyle A_i \binom{u}{v}$ is in $\bZ^2$ and put
\begin{equation} \label{D6 lattice det} m_i = d(\Lambda_i) \end{equation}
for $i = 1,2,3,4$. 

\begin{theorem} \label{WF value theorem} The positive number $C_F$ in the statement of Theorem \ref{Rep by BF MT} is equal to $W_F A_F$ where $A_F$ is given by (\ref{AF defn}) and $W_F$ is given by the following table: 

\begin{center} \label{Table 2}
\begin{tabular}
{ |p{1.5cm}|p{3.3cm}|p{1.5cm}|p{7.7cm}|  }
 
 \hline
$\Rep (F)$  & $W_F$ & $\Rep (F)$ & $W_F$ \\
 \hline
 $\CC_1$   & $1$    & $\DD_1$ &   $1 - \dfrac{1}{2m}$ \\ & & & \\
 $\CC_2$ &   $\dfrac{1}{2}$  & $\DD_2$ & $\dfrac{1}{2} \left(1 - \dfrac{1}{2m}\right)$ \\ & & & \\
 $\CC_3$ & $1 - \dfrac{2}{3m}$ & $\DD_3$ & $1 - \dfrac{1}{2m_1} - \dfrac{1}{2m_2} - \dfrac{1}{2m_3} - \dfrac{2}{3m_4} + \dfrac{4}{3m} $\\ & & & \\
 $\CC_4$    & $\dfrac{1}{2} \left(1 - \dfrac{1}{2m}\right)$ & $\DD_4$ &  $\dfrac{1}{2} \left(1 - \dfrac{1}{2m_1} - \dfrac{1}{2m_2} - \dfrac{1}{2m_3} + \dfrac{3}{4m}\right)$ \\ & & & \\
 $\CC_6$ &  $\dfrac{1}{2} \left(1 - \dfrac{2}{3m}\right)$  & $\DD_6$ & $\dfrac{1}{2} \left(1 - \dfrac{1}{2m_1} - \dfrac{1}{2m_2} - \dfrac{1}{2m_3} - \dfrac{2}{3m_4} + \dfrac{4}{3m}\right)$ \\ & & & \\
  \hline 
\end{tabular}
\end{center}
Here $\Rep(F)$ denotes a representative of the equivalence class of $\Aut F$ under $\GL_2(\bQ)$ conjugation and $m, m_1, m_2, m_3, m_4$ are defined in (\ref{lattice det}),(\ref{D3 lattice det}), (\ref{D4 lattice det}), and (\ref{D6 lattice det}). \end{theorem}

We remark, see Lemma \ref{redundancy lemma}, that if $\Aut F$ is equivalent to $\DD_4$ then $m = \lcm(m_1, m_2, m_3)$, the least common multiple of $m_1, m_2$ and $m_3$, and if $\Aut F$ is equivalent to $\DD_3$ or $\DD_6$ then $m = \lcm(m_1, m_2, m_3, m_4)$. \\

Observe that if $F$ is a binary form with $F(1,0) \ne 0$ and $A = \begin{pmatrix} a_1 & a_2 \\ a_3 & a_4 \end{pmatrix}$ is in $\Aut F$ then $A$ acts on the roots of $F$ by sending a root $\alpha$ to $\dfrac{a_1 \alpha + a_2}{a_3 \alpha + a_4}$. If $A$ fixes a root $\alpha$ then
\[ a_3 \alpha^2 + (a_4 - a_1)\alpha + a_2 = 0.\]
If $F$ is an irreducible cubic then $\alpha$ has degree $3$ and so 
\[a_3 = a_4 - a_1 = a_2 = 0,\]
hence
\[A = \begin{pmatrix} 1 & 0 \\ 0 & 1 \end{pmatrix} \text{ or } A = \begin{pmatrix} -1 & 0 \\ 0 & -1 \end{pmatrix}.\]
But since $F$ has degree $3$ we see that $A = \begin{pmatrix} 1 & 0 \\ 0 & 1 \end{pmatrix}$. Therefore the only element of $\Aut F$ which fixes a root of $F$ is the identity matrix $I$. \\

If $A$ in $\Aut F$ does not fix a root it must permute the roots cyclically and thus must have order $3$. Further, since any element in $\Aut F$ of order $2$ would fix a root of $F$, we find that $\Aut F$ is $\GL_2(\bQ)$-conjugate to $\CC_3$, say $\Aut F = T \CC_3 T^{-1}$ with $T$ in $\GL_2(\bQ)$. Forms invariant under $\CC_3$ are of the form
\[G(x,y) = ax^3 + bx^2 y + (b - 3a)xy^2 - ay^3\]
with $a$ and $b$ integers; see (74) of \cite{ST1}. Notice that
\[\Delta(G) = (b^2 - 3ab + 9a^2)^2.\]
Then $F = G_T$ for some $G$ invariant under $\CC_3$ and so
\[\Delta(F) = (\det T)^6 \Delta(G).\]
We conclude that if $F$ is an irreducible cubic form with discriminant not a square then $\Aut F$ is $\CC_1$ and so $W_F = 1$; thus Hooley's result (\ref{Hooley result 1}) follows from Theorems \ref{Rep by BF MT} and \ref{WF value theorem}. When $\Aut F$ is equivalent to $\CC_3$ then $W_F = 1 - \dfrac{2}{3m}$ where $m$ is the determinant of the lattice consisting of $(u,v)$ in $\bZ^2$ for which $\displaystyle A \binom{u}{v}$ is in $\bZ^2$ for all $A$ in $\Aut F$. By Lemma \ref{redundancy lemma} it suffices to consider the lattice consisting of $(u,v)$ in $\bZ^2$ for which $\displaystyle A \binom{u}{v}$ is in $\bZ^2$ for a generator $A$ of $\Aut F$. Hooley has shown in \cite{Hoo2} that the determinant of the lattice is $m$ and so (\ref{Hooley result 2}) follows from Theorems \ref{Rep by BF MT} and \ref{WF value theorem}. \\

Now if $F(x,y) = ax^4 + bx^2 y^2 + cy^4$ and the discriminant of $F$ is non-zero then $\Aut F$ is equivalent to $\DD_2$ unless $a/c = A^4 / C^4$ with $A$ and $C$ coprime positive integers. In this case $\Aut F$ is equivalent to $\DD_4$. In the first instance $m = 1$ and $W_F = \dfrac{1}{4}$ and so we recover Hooley's estimate (\ref{Hooley result 3}). In the second case $m_1 = 1$ and $m_2 = m_3 = m = AC$ and so 
\[W_F = \frac{1}{4} \left(1 - \frac{1}{2AC}\right).\]
which gives (\ref{Hooley result 4}). \\

If follows from the analysis on page 818 of \cite{ST1} that when $F$ is a binary cubic form with non-zero discriminant $\Aut F$ is equivalent to $\CC_1, \CC_3, \DD_1$ or $\DD_3$ whereas if $F$ is a binary quartic form with non-zero discriminant $\Aut F$ is equivalent to $\CC_2, \CC_4, \DD_2$ or $\DD_4$. In \cite{X} and \cite{XThesis} the second author gives a set of generators for $\Aut F$ in these cases and as a consequence it is possible to determine $W_F$ explicitly in terms of the coefficients of $F$. \\

In the special case that $F$ is a binomial form, so $F(x,y) = ax^d + by^d$, it is straightforward to determine $\Aut F$; see Lemma \ref{binomial lemma}. Then, by Theorems \ref{Rep by BF MT} and \ref{WF value theorem}, we have the following result.

\begin{corollary} \label{binomial} Let $a, b$ and $d$ be non-zero integers with $d \geq 3$ and let 
\[F(x,y) = ax^d + by^d.\]
Then (\ref{rep by BF ME}) holds with $C_F = W_F A_F$.  If $d$ is even and $ab>0$ then $\beta_F$ is given by (\ref{Beta DD}). If $a/b$ is not the $d$-th power of a rational number then
\[W_F = \begin{cases} 1 & \text{if } d \text{ is odd,} \\ \\
\dfrac{1}{4} & \text{if } d \text{ is even.} \end{cases}\]

If $\dfrac{a}{b} = \left(\dfrac{A}{B}\right)^d$ with $A$ and $B$ coprime integers then
\[W_F = \begin{cases} 1 - \dfrac{1}{2|AB|} & \text{if } d \text{ is odd,} \\
\dfrac{1}{4} \left(1 - \dfrac{1}{2 |AB|} \right) & \text{if } d \text{ is even.} \end{cases}\]
Further if $d$ is odd then
\[A_F = \frac{1}{d |ab|^{1/d}} \left(\dfrac{2 \Gamma(1 - 2/d) \Gamma(1/d)}{\Gamma(1 - 1/d)} + \frac{\Gamma^2(1/d)}{\Gamma(2/d)} \right)\]
while if $d$ is even
\[A_F = \frac{2}{d|ab|^{1/d}} \frac{\Gamma^2(1/d)}{\Gamma(2/d)} \quad \text{ if } ab > 0\]
and
\[A_F = \frac{4}{d|ab|^{1/d}} \frac{\Gamma(1/d) \Gamma(1 - 2/d)}{\Gamma(1 - 1/d)} \quad \text{ if } ab < 0.\]
\end{corollary}
$$
$$

Finally we mention that there are other families of forms where one may readily determine $W_F$. For instance let  $a, b$ and $k$ be integers with $a \ne 0$, $2a \ne \pm b$ and $k \geq 2$ and put 
\begin{equation} \label{Thue} F(x,y)=ax^{2k}+bx^ky^k+ay^{2k} .\end{equation}
The discriminant of $F$ is non-zero since $a\ne0$ and $2a \ne \pm b$. Further, $\DD_4$ is plainly contained in $\Aut F$ and there is no larger group which is an automorphism group of a binary form which contains  $\DD_4$. Therefore  $\DD_4$ is  $\Aut F$. It now follows from Theorem \ref{WF value theorem} that $W_F=1/8$ since  $m_1 = m_2 = m_3 = m = 1$.

\section{Preliminary lemmas}
\label{Rep BF prelim}

We shall require a result of Mahler \cite{Mah} from 1933. For a positive number $Z$ we put
\[\N_F(Z) = \{(x,y) \in \bZ^2 : 0 < |F(x,y)| \leq Z \}\]
and
\[N_F(Z) = |\N_F(Z)|.\]

\begin{lemma} \label{Mahler Thm} Let $F$ be a binary form with integer coefficients, non-zero discriminant and degree $d \geq 3$. Then, with $A_F$ defined by (\ref{AF defn}), we have
\[  N_F(Z)  = A_F Z^{\frac{2}{d}} + O_F \left(Z^{\frac{1}{d-1}}\right).\]
\end{lemma}

In fact Mahler proved this result only under the assumption that $F$ is irreducible. However, Lemma \ref{Mahler Thm} can be deduced as a special case of Theorem 3 in \cite{Thun}. 

\begin{lemma} \label{HB lemma} Let $F$ be a binary form with integer coefficients, non-zero discriminant and degree $d \geq 3$. Let $Z$ be a positive real number and let $\gamma$ be a real number larger than $1/d$. The number of pairs of integers $(x,y)$ with
\begin{equation} \label{thue bound} 0 < |F(x,y)| \leq Z \end{equation}
for which
\[\max\{|x|, |y|\} > Z^\gamma\]
is 
\[O_F \left(Z^{\frac{1}{d}} \log Z + Z^{1 - (d-2) \gamma} \right).\]
\end{lemma}
\begin{proof} We shall follow Heath-Brown's proof of Theorem 8 in \cite{HB1}. Accordingly put
\[S(Z;C) = |\{(x,y) \in \bZ^2 : 0 < |F(x,y)| \leq Z, C < \max\{|x|, |y|\} \leq 2C, \gcd(x,y) = 1\}|\]
and suppose that $C \geq Z^\gamma$. Heath-Brown observes that by Roth's theorem $S(Z;C) = 0$ unless $C \ll Z^2$. Further, 
\begin{equation} \label{HB bound 1} S(Z;C) \ll 1 + \frac{Z}{C^{d-2}}.
\end{equation}
Put
\[S^{(1)}(Z;C) = |\{(x,y) \in \bZ^2 : 0 < |F(x,y)| \leq Z, C < \max\{|x|,|y|\}, \gcd(x,y) = 1\}|.\]
Therefore, on replacing $C$ by $2^j C$ in (\ref{HB bound 1}) for $j= 1,2,...$ and summing we find that
\[S^{(1)}(Z;C) \ll \log Z + \frac{Z}{C^{d-2}}.\]
Next put 
\[S^{(2)}(Z;C) = |\{(x,y) \in \bZ^2 : 0 < |F(x,y)| \leq Z, C < \max\{|x|, |y|\}|.\]
Then
\begin{align*} S^{(2)}(Z;C) & \ll \sum_{h \leq Z^{1/d}} S^{(1)} \left(\frac{Z}{h^d}, \frac{C}{h} \right) \\
& \ll \sum_{h \leq Z^{1/d}} \left(\log Z + \frac{Z}{h^2C^{d-2}} \right) \\
& \ll Z^{\frac{1}{d}} \log Z + \frac{Z}{C^{d-2}}
\end{align*}
and our result follows on taking $C = Z^\gamma$. \\

We note that instead of appealing to Roth's theorem it is possible to treat the large solutions of (\ref{thue bound}) by means of the Thue-Siegel principle; see \cite{BS} and \cite{ST1}. As a consequence all constants in the proof are then effective. \end{proof}

We say that an integer $h$ is \emph{essentially represented} by $F$ if whenever $(x_1, y_1), (x_2, y_2)$ are in $\bZ^2$ and
\[F(x_1, y_1) = F(x_2, y_2) = h\]
then there exists $A$ in $\Aut F$ such that
\[A \binom{x_1}{y_1} = \binom{x_2}{y_2}.\]
Observe that if there is only one integer pair $(x_1,y_1)$ for which $F(x_1,y_1)=h$ then $h$ is essentially represented since $I$ is in $\Aut F$. \\

Put
\[\N_F^{(1)}(Z) = \{(x,y) \in \bZ^2 : 0 < |F(x,y)| \leq Z \text{ and } F(x,y) \text{ is essentially represented by } F\}\]
and
\[\N_F^{(2)}(Z) = \{(x,y) \in \bZ^2 : 0 < |F(x,y)| \leq Z \text{ and } F(x,y) \text{ is not essentially represented by } F\}.\]
Let $N_F^{(i)}(Z) = |\N_F^{(i)}(Z)|$ for $i = 1,2$. \\

Let $X$ be a smooth surface in $\bP^3$ of degree $d$ defined over $\bQ$, and for a positive number $B$ let $N_1(X; B)$ denote the number of  integer points on $X$ with height at most $B$ which do not lie on any lines contained in $X$. Colliot-Th\'el\`ene proved in the appendix of \cite{HB1} that if $X$ is a smooth projective surface of degree $d \geq 3$ then there are at most $O_d(1)$  curves of degree at most $d-2$ contained in $X$. This, combined with Salberger's work in \cite{S3}, implies that for any $\ep > 0$, we have
\begin{equation} \label{Salberger 1} N_1(X; B) = O_\ep \left(B^{\frac{12}{7} + \ep} \right) \text{ if } d = 3.\end{equation}
Heath-Brown obtained a better estimate for $N_1(X;B)$ when $d = 3$ and the surface $X$ contains three lines which are rational and co-planar in \cite{HB3}. In particular, he proved that in this case we have
\begin{equation} \label{Heath-Brown 1} N_1(X;B) = O_{\ep} \left(B^{\frac{4}{9} + \ep}\right). \end{equation} 
Further, by the main theorem of the global determinant method for projective surfaces of Salberger \cite{S2}, which has been generalized to the case of weighted projective space in Theorem 3.1 of \cite{X0}, and controlling the contribution from conics contained in a projective surface $X$, as was done by Salberger in \cite{S1}, we obtain
\begin{equation} \label{Salberger 2} N_1(X;  B) = O_{d,\ep} \left(B^{\frac{3}{\sqrt{d}} + \ep}  + B^{1 + \ep} \right) \text{ if } d \geq 4.\end{equation} 
To make use of (\ref{Heath-Brown 1}), we shall require the following lemma, which is a consequence of a result on characterizing lines on surfaces $X$ of the shape
\[X: F(x_1, x_2) - F(x_3, x_4) = 0\]
for a binary form $F$ with $\deg F \in \{3,4\}$ in \cite{X}. 

\begin{lemma} \label{red cubics} Let $F$ be a binary cubic form with integer coefficients and non-zero discriminant. Let $X$ be the surface in $\bP^3$ given by the equation 
\[F(x_1, x_2) - F(x_3, x_4) = 0.\]
Then $X$ contains three rational, co-planar lines if $F$ is reducible over $\bQ$.  
\end{lemma}

\begin{proof} We first show that $X$ contains three rational, co-planar lines if $F$ has a rational automorphism of order 2. Since all elements of order 2 in $\GL_2(\bQ)$ are $\GL_2(\bQ)$-conjugate to $\left(\begin{smallmatrix} 0 & 1 \\ 1 & 0 \end{smallmatrix}\right)$ and the property of $X$ having three rational, co-planar lines is preserved under $\GL_2(\bQ)$-transformations of $F$, we may assume that $T = \left(\begin{smallmatrix} 0 & 1 \\ 1 & 0 \end{smallmatrix}\right) \in \Aut F$. In particular, we assume that $F$ is symmetric; an elementary calculation shows that $F$ is divisible by the linear form $x+y$. By Lemma 5.2 in \cite{X} we see that $X(\bR)$ contains the lines 
\[\{[s,t,s,t] : s,t \in \bR\}, \{[s,t,t,s] : s,t, \in \bR\}, \{[s,-s,t,-t] : s,t \in \bR\}\]
in $\bP^3(\bR)$. These lines all lie in the plane given by the equation 
\[x_1 + x_2 - x_3 - x_4 = 0. \]
Each of these three lines is plainly rational, hence $X$ contains three rational, co-planar lines. Now Theorem 3.1 in \cite{X} gives that $F$ is reducible if and only if $\Aut F$ contains an element of order 2, which completes the proof. \end{proof}

\begin{lemma} \label{EFZ lemma} Let $F$ be a binary form with integer coefficients, non-zero discriminant and degree $d \geq 3$. Then, for each $\ep > 0$, 
\[N_F^{(2)}(Z) = O_{F,\ep} \left(Z^{\beta_F + \ep}\right),\]
where $\beta_F$ is given by (\ref{Beta D}) and (\ref{Beta DD}). 

\end{lemma}

\begin{proof} Let $\ep >0$. If $F$ has a linear factor over $\bR$ put

\[\eta = \begin{cases} \dfrac{7}{19} & \text{if } d = 3 \text{ and } F \text{ is irreducible}, \\ \\
\dfrac{9}{13} & \text{if } d = 3 \text{ and } F \text{ is reducible}, \\ \\
 \dfrac{\sqrt{d}}{d\sqrt{d} - 2 \sqrt{d} + 3} &\text{if } 4 \leq d \leq 8, \\ \\ \dfrac{1}{d-1} & \text{if } d \geq 9. \end{cases} \]
Otherwise put 
\[\eta = \frac{1}{d}+ \ep.\]

We will give an upper bound for $N_F^{(2)}(Z)$ by following the approach of Heath-Brown in his proof of Theorem 8 of \cite{HB1}. We first split  $\N_F^{(2)}(Z)$ into two sets:
\begin{enumerate}
\item Those points $(x,y) \in \N_F^{(2)}(Z)$ which satisfy $\max\{|x|, |y|\} \leq Z^{\eta},$ 

and

\item Those points $(x,y) \in \N_F^{(2)}(Z)$ which satisfy $\max\{|x|, |y|\} > Z^{\eta}.$ 
\end{enumerate}
We will use (\ref{Salberger 1}), (\ref{Heath-Brown 1}), and (\ref{Salberger 2}) to treat the points in category (1). Let us put
\[\G(\Bx) = F(x_1, x_2) - F(x_3, x_4).\]
We shall denote by $X$  the surface defined by $\G(\Bx) = 0$. Notice that $X$ is smooth since $\Delta(F)\neq0.$ \\

Let $N_2(X;B)$ be the number of integer points $(r_1, r_2, r_3, r_4)$ in $\bR^4$ with $\displaystyle \max_{1 \leq i \leq 4} |r_i| \leq B$  for which $(r_1,r_2,r_3,r_4)$, viewed as a point in $\bP^3$, is on $X$ but does  not lie on a line in $X$; here we do not require $\gcd(r_1, r_2, r_3, r_4) = 1$. Then
\[N_2(X;B) \leq \sum_{t=1}^B N_1\left(X; \frac{B}{t} \right)\]
and so, by (\ref{Salberger 1}), (\ref{Heath-Brown 1}), (\ref{Salberger 2}), and Lemma \ref{red cubics}, 
\[N_2(X;B) = O_\ep \left(B^{\frac{12}{7} + \ep} \right) \quad \text{ if } d = 3 \text{ and } F \text{ is irreducible},\]
\[N_2(X;B) = O_\ep \left(B^{\frac{4}{9} + \ep}\right) \quad \text{ if } d = 3 \text{ and } F \text{ is reducible},\]
and
\[N_2(X;B) = O_{d,\ep} \left(B^{\frac{3}{\sqrt{d}} + \ep} + B^{1 + \ep} \right) \quad \text{ if } d \geq 4.\]
Therefore 
\begin{equation} \label{N2 bound} N_2(X; Z^{\eta}) = O_{d, \ep} \left( Z^{\beta_F + \ep} \right). \end{equation}

It remains to deal with integer points on $X$ which lie on some line contained in $X$. Lines in $\bP^3$ may be classified into two types. They are given by the pairs
\[u_1 x_1 + u_2 x_2 + u_3 x_3 + u_4 x_4 = 0, v_3 x_3 + v_4 x_4 = 0,\]
and by
\[x_1 = u_1 x_3 + u_2 x_4, x_2 = u_3 x_3 + u_4 x_4.\]
Suppose the first type of line is on $X$. Then one of $v_3, v_4$ is non-zero, and we may assume without loss of generality that $v_3 \ne 0$. We thus have
\[x_3 = \frac{-v_4}{v_3} x_4.\]
Substituting this back into the first equation yields
\[u_1 x_1 + u_2 x_2 = -u_3 \frac{-v_4}{v_3} x_4 - u_4 x_4 = \frac{u_3 v_4 - v_3 u_4}{v_3} x_4.\]
Substituting this back into $F(x_1, x_2) = F(x_3, x_4)$ and assuming that $u_3 v_4 - v_3 u_4 \ne 0$, we see that
\begin{align*} F(x_1, x_2) & = F\left(\frac{-v_4}{v_3} x_4, x_4\right) = x_4^d F(-v_4/v_3, 1) \\
& = F\left(\frac{-v_4}{v_3}, 1 \right) \left(\frac{v_3 u_1}{u_3 v_4 - v_3 u_4} x_1 + \frac{u_2 v_3}{u_3 v_4 - u_4 v_3} x_2\right)^d.
\end{align*}
If $F(-v_4/v_3, 1) \ne 0$, then we see that $F$ is a perfect $d$-th power, which is not possible since $\Delta(F) \ne 0$. Therefore we must have $F(x_1, x_2) = 0$ which is a contradiction. Now suppose that $u_3 v_4 = v_3 u_4$. We see that $u_1, u_2$ cannot both be zero. Assume without loss of generality that $u_1 \ne 0$. Then 
\[F(x_1, x_2) = x_2^d F(-u_2/u_1, 1),\] which is not possible since $\Delta(F) \ne 0$. Therefore we must have $F(-u_2/u_1, 1) = 0$, so once again $F(x_1, x_2) = 0$. \\ 

Now suppose that $X$ contains a line of the second type. Suppose that $u_1 u_4 = u_2 u_3$. Since at least one of $u_1, u_2$ and one of $u_3, u_4$ is non-zero, we may assume that $u_1$ and $u_3$ are non-zero. Then we have
\[u_3 x_1 = u_1 u_3 x_3 + u_2 u_3 x_4 = u_1 (u_3 x_3 + u_4 x_4),\]
hence
\[(u_3/u_1)x_1 = u_3 x_3 + u_4 x_4 = x_2.\]
Thus, $F(x_1,x_2) = F(x_3,x_4)$ implies that
\[F(x_3, x_4) = x_1^d F(1,u_3/u_1) = (u_3 x_3 + u_4 x_4)^d(u_1/u_3)^d F(1,u_3/u_1).\]
As before we must have $F(x_3, x_4) = 0$. \\ 

The last case is a line of the second type and for which $u_1 u_4 \ne u_2 u_3$. Such a line yields the equation
\[F(x_3, x_4) = F(u_1 x_3 + u_2 x_4, u_3 x_3 + u_4 x_4).\]
If $(r_1, r_2,r_3,r_4)$ is an integer point on $X$ on such a line and there is no element $A$ of $\Aut F$ which maps $(r_1,r_2)$ to $(r_3,r_4)$ then it follows that at least one of $u_1, u_2, u_3$ and $u_4$ is  not rational. Therefore, $\Bx = (x_1, x_2, x_3, x_4)$ must lie on a line which is not defined over $\bQ$ and hence has at most one primitive integer point on it. Thus there are at most $O\left(Z^{\eta}\right)$ integer points whose coordinates have absolute value at most $Z^{\eta}$ which lie on it. Since $X$ is smooth it follows from a classical result of Salmon and Clebsch, see p. 559 of \cite{Sal} or \cite{BSa}, that there are at most $O_{d}(1)$ lines on $X$ and so at most $O_{d}(Z^{\eta})$ integer points whose coordinates have absolute value at most $Z^{\eta}$ on lines on $X$ which are not defined over $\bQ$. This, together with (\ref{N2 bound}), shows that the number of points in category (1) is at most
\[O_{d,\ep} \left(Z^{\beta_F + \ep}\right).\]

When $F$ has a linear factor over $\bR$ we apply Lemma \ref{HB lemma} with $\gamma = \eta$ to conclude that  the number of points in category (2) is at most $O_{F,\ep}\left(Z^{\beta_F + \ep}\right)$. Otherwise we may write  
\[F(x,y)= \prod_{j=1}^d L_j(x,y)\]
with say $L_j(x,y)=\lambda_jx+\theta_jy$ where $\lambda_j$ and $\theta_j$ are non-zero complex numbers whose ratio is not a real number. But then
\[ |L_j(x,y)| \gg_{\lambda_j,\theta_j} \max(|x|,|y|)\]
and so
\[ |F(x,y)| \gg_{F} \max(|x|,|y|)^d.\]
Therefore, in this case the number of points in category (2) is at most $O_{F,\ep}(1)$ and the result now follows.\end{proof} 
 
In \cite{HB1} Heath-Brown proved that for each $\ep > 0$ the number of integers $h$ of absolute value at most $Z$ which are not essentially represented by $F$ is
\begin{equation} \label{N22 bound}   O_{F,\ep} \left(Z^{\frac{12d + 16}{9d^2 - 6d + 16} + \ep}\right),  \end{equation}
whenever $F$ is a binary form with integer coefficients and non-zero discriminant. This follows from the remark on page 559 of \cite{HB1} on noting that the numerator of the exponent should be $12d + 16$ instead of $12d$. Observe that the exponent is less than $2/d$ for $\ep$ sufficiently small. It follows from (\ref{N22 bound}) that Lemma \ref{EFZ lemma} holds with $\beta_d$ replaced by the larger quantity given by the exponent of $Z$ in (\ref{N22 bound}). To see this we denote, for any positive integer $h$, the number of prime factors of $h$ by $\omega(h)$ and the number of positive integers which divide $h$ by $\tau(h)$. By Bombieri and Schmidt \cite{BS} when $F$ is irreducible and by Stewart \cite{ST1} when $F$ has non-zero discriminant, if $h$ is a non-zero integer the Thue equation
\begin{equation} \label{Thueequation} F(x,y)=h, \end{equation}
has at most $2800d^{1+\omega(h)}$ solutions in coprime integers $x$ and $y$. Therefore the number of solutions of (\ref{Thueequation}) in integers $x$ and $y$ is at most
\begin{equation} \label{Thueestimate} 2800\tau(h)d^{1+\omega(h)}. \end{equation}
Our claim now follows from (\ref{N22 bound}), (\ref{Thueestimate}) and Theorem 317 of \cite{HW}.

\begin{lemma} \label{essential reps lemma} Let $F$ be a binary form with integer coefficients, non-zero discriminant and degree $d \geq 3$. Then with $A_F$ defined as in (\ref{AF defn}), 
\[N_F^{(1)}(Z) = A_F Z^{\frac{2}{d}} + O_{F,\ep} \left(Z^{\beta_F + \ep} \right)\]
where $\beta_F$ is given by (\ref{Beta D}) and (\ref{Beta DD}).

\end{lemma}

\begin{proof} This is an immediate consequence of Lemmas \ref{Mahler Thm} and \ref{EFZ lemma} since $1/(d-1)$ is less than or equal to $\beta_F$ .\end{proof}

\section{The automorphism group of $F$ and associated lattices}
\label{auto group and lattices}

For any element $A$ in $\GL_2(\bQ)$ we denote by $\Lambda(A)$ the lattice of $(u,v)$ in $\bZ^2$ for which $\displaystyle A \binom{u}{v}$ is in $\bZ^2$. 

\begin{lemma} \label{gen det char} Let $F$ be a binary form with integer coefficients and non-zero discriminant. Let $A$ be in $\Aut F$. Then there exists a unique positive integer $a$ and coprime integers $a_1, a_2, a_3, a_4$ such that
\begin{equation} \label{primitive form} A = \frac{1}{a} \begin{pmatrix} a_1 & a_2 \\ a_3 & a_4 \end{pmatrix},\end{equation}
and
\begin{equation} \label{gen det} a = d (\Lambda(A)).\end{equation}

\end{lemma}

\begin{proof} If $A = \begin{pmatrix} \alpha_1 & \alpha_2 \\ \alpha_3 & \alpha_4 \end{pmatrix}$ is in $\GL_2(\bQ)$, we write 
\[\alpha_i = \frac{a_i}{a}\]
for $i = 1, 2, 3, 4$ where $a$ is the least common denominator of the $\alpha_i$'s. This yields the form given in (\ref{primitive form}). Then $\Lambda(A)$ is the set of $(u,v)$ in $\bZ^2$ for which 
\[a_1 u + a_2 v \equiv 0 \pmod{a}\]
and
\[a_3 u + a_4 v \equiv 0 \pmod{a}.\]

For each prime $p$ let $k$ be the largest power of $p$ which divides $a$. We define the lattice $\Lambda^{(p)}(A)$ to be the set of $(u,v)$ in $\bZ^2$ for which
\begin{equation} \label{generic congruence 1} a_1 u + a_2 v \equiv 0 \pmod{p^k}
\end{equation}
and
\begin{equation} \label{generic congruence 2} a_3 u + a_4 v \equiv 0 \pmod{p^k}. \end{equation}
Then 
\begin{equation} \label{p-adic intersect} \Lambda(A) = \bigcap_p \Lambda^{(p)}(A),\end{equation}
where the intersection is taken over all primes $p$, or equivalently over primes $p$ which divide $a$. \\

Since $a_1, a_2, a_3$ and $a_4$ are coprime at least one of them is not divisible by $p$. Suppose, without loss of generality, that $p$ does not divide $a_1$. Then $a_1^{-1}$ exists modulo $p^k$. Thus if (\ref{generic congruence 1}) holds then 
\[u \equiv - a_1^{-1} a_2 v \pmod{p^k}\]
and (\ref{generic congruence 2}) becomes
\begin{equation} \label{redundant cong 1} (a_1 a_4 - a_2 a_3)v \equiv 0 \pmod{p^k}. \end{equation}

But $A$ is in $\Aut F$ and so $|\det(A)| = 1$. Thus 
\[|a_1 a_4 - a_2 a_3| = a^2\]
and (\ref{redundant cong 1}) holds regardless of the value of $v$. Therefore the elements of the lattice $\Lambda^{(p)}(A)$ are determined by the congruence relation (\ref{generic congruence 1}). It follows that 
\[d (\Lambda^{(p)} (A)) = p^k\]
and by (\ref{p-adic intersect}) and the Chinese Remainder Theorem we obtain (\ref{gen det}). 
\end{proof}

\begin{lemma} \label{redundancy lemma} Let $F$ be a binary form with integer coefficients, non-zero discriminant and degree $d \geq 3$. If $A$ is an element of order $3$ in $\Aut F$ then 
\begin{equation} \label{C3 C6} \Lambda(A) = \Lambda(A^2).\end{equation}
If $\Aut F$ is equivalent to $\DD_3, \DD_4$ or $\DD_6$ then 
\begin{equation} \label{dihedral intersect} \Lambda_i \cap \Lambda_j = \Lambda \text{ for } i \ne j. \end{equation}
Further $m = \lcm(m_1, m_2, m_3)$ when $\Aut F$ is equivalent to $\DD_4$ and $m = \lcm(m_1, m_2, m_3, m_4)$ when $\Aut F$ is equivalent to $\DD_3$ or $\DD_6$.

\end{lemma}

\begin{proof} Let us first prove (\ref{C3 C6}). Then either $A$ or $A^2$ is conjugate in $\GL_2(\bQ)$ to $\begin{pmatrix} 0 & 1 \\ -1 & -1 \end{pmatrix}$ and we may assume we are in the former case. Let $T$ be an element of $\GL_2(\bQ)$ with
\begin{equation} \label{transfer matrix} T = \begin{pmatrix} t_1 & t_2 \\ t_3 & t_4 \end{pmatrix},\end{equation}
where $t_1, t_2, t_3$ and $t_4$ are coprime integers for which 
\begin{equation} \label{conjugated matrix} A = T^{-1} \begin{pmatrix} 0 & 1 \\ -1 & -1 \end{pmatrix} T. \end{equation}
Put $t = t_1 t_4 - t_2 t_3$. Then
\[A = \frac{1}{t} \begin{pmatrix} t_1 t_2 + t_2 t_3 + t_3 t_4 & t_2^2 + t_4^2 + t_2 t_4 \\ -t_1 t_3 - t_3^2 - t_1^2 & -t_1 t_4 - t_3 t_4 - t_1 t_2 \end{pmatrix}\]
and
\[A^2 = \frac{1}{t} \begin{pmatrix} - t_1 t_2 - t_3 t_4 - t_1 t_4 & - t_2^2 - t_4^2 - t_2 t_4 \\ t_1^2 + t_3^2 + t_1 t_3 & t_1 t_2 + t_3 t_4 + t_2 t_3 \end{pmatrix},\]
hence $\Lambda(A)$ is the set of $(u,v) \in \bZ^2$ for which
\begin{equation} \label{C3 cong 1} (t_1 t_2 + t_2 t_3 + t_3 t_4) u + (t_2^2 + t_4^2 + t_2 t_4) v \equiv 0 \pmod{t} \end{equation} 
and
\begin{equation} \label{C3 cong 2} (t_1 t_3 + t_3^2 + t_1^2) u + (t_1 t_4 + t_3 t_4 + t_1 t_2) v \equiv 0 \pmod{t}. \end{equation}
Similarly, $\Lambda(A^2)$ is the set of $(u,v) \in \bZ^2$ for which
\begin{equation} \label{C3 cong 3} (t_1 t_2 + t_1 t_4 + t_3 t_4) u + (t_2^2 + t_4^2 + t_2 t_4)v \equiv 0 \pmod{t} \end{equation}
and
\begin{equation} \label{C3 cong 4} (t_1^2 + t_3^2 + t_1 t_3) u + (t_2 t_3 + t_3 t_4 + t_1 t_2)v \equiv 0 \pmod{t}. \end{equation}

On noting that $t_1 t_4 \equiv t_2 t_3 \pmod{t}$ we see that the conditions (\ref{C3 cong 1}) and (\ref{C3 cong 2}) are the same as (\ref{C3 cong 3}) and (\ref{C3 cong 4}), hence 
\[\Lambda(A) = \Lambda(A^2).\]  

Suppose that $\Aut F$ is equivalent to $\DD_4$ under conjugation in $\GL_2(\bQ)$. Then there exists an element $T$ in $\GL_2(\bQ)$ given by (\ref{transfer matrix}) with $t_1, t_2, t_3$ and $t_4$ coprime integers for which $\Aut F = T^{-1} \DD_4 T$. Put $t = t_1 t_4 - t_2 t_3$ and note that $t \ne 0$. The lattices $\Lambda_1, \Lambda_2$ and $ \Lambda_3$ may be taken to be the lattices of $(u,v)$ in $\bZ^2$ for which 
\[T^{-1} A_i T \binom{u}{v} \in \bZ^2,\]
where 
\[A_1 = \begin{pmatrix} 0 & 1 \\ -1 & 0 \end{pmatrix}, A_2 = \begin{pmatrix} 0 & 1 \\ 1 & 0 \end{pmatrix}, A_3 = \begin{pmatrix} 1 & 0 \\ 0 & -1 \end{pmatrix}.\]
Thus $\Lambda_1$ consists of integer pairs $(u,v)$ for which
\begin{equation} \label{D4 cong 1} (t_1 t_2 + t_3 t_4)u + (t_2^2 + t_4^2) v \equiv 0 \pmod{t} 
\end{equation}

and
\begin{equation} \label{D4 cong 2} (t_1^2 + t_3^2) u + (t_1 t_2 + t_3 t_4)v \equiv 0 \pmod{t}. \end{equation}
$\Lambda_2$ consists of integer pairs $(u,v)$ for which
\begin{equation} \label{D4 cong 3} (t_1 t_2 - t_3 t_4)u + (t_2^2 - t_4^2)v \equiv 0 \pmod{t} \end{equation} 

and
\begin{equation} \label{D4 cong 4} (t_1^2 - t_3^2)u + (t_1 t_2 - t_3 t_4)v \equiv 0 \pmod{t} \end{equation} 
and $\Lambda_3$ consists of integer pairs $(u,v)$ for which 
\begin{equation} \label{D4 cong 5} 2t_2 t_3 u + 2t_2 t_4 v \equiv 0 \pmod{t} \end{equation}

and
\begin{equation} \label{D4 cong 6} 2t_1 t_3 u + 2 t_2 t_3 v \equiv 0 \pmod{t}, \end{equation}
where in (\ref{D4 cong 5}) and (\ref{D4 cong 6}) we have used the observation that
\[t_1 t_4 \equiv t_2 t_3 \pmod{t}.\]

For each prime $p$ dividing $t$ we put $h = \ord_p t$. Define $\Lambda_i^{(p)}$ for $i = 1,2,3$ to be the lattice of $(u,v)$ in $\bZ^2$ for which the congruences (\ref{D4 cong 1}) and (\ref{D4 cong 2}), (\ref{D4 cong 3}) and (\ref{D4 cong 4}), and (\ref{D4 cong 5}) and (\ref{D4 cong 6}) respectively hold with $t$ replaced by $p^h$ and define $\Lambda^{(p)}$ to be the lattice for which all of the congruences hold. We shall prove that for some reordering $(i, j, k)$ of $(1,2,3)$ we have
\begin{equation} \label{D4 pairwise intersect} \Lambda_{i}^{(p)} \supset \Lambda_{j}^{(p)} = \Lambda_{k}^{(p)}. \end{equation}
It then follows that
\begin{equation} \label{D4 total intersect} \Lambda_{r}^{(p)} \cap \Lambda_s^{(p)} = \Lambda_1^{(p)} \cap \Lambda_2^{(p)} \cap \Lambda_3^{(p)} = \Lambda^{(p)} \end{equation}
for any pair $\{r,s\}$ from $\{1,2,3\}$. But since
\begin{equation} \label{D4 p-adic intersect} \bigcap_p \left(\Lambda_r^{(p)} \cap \Lambda_s^{(p)} \right) = \Lambda_r \cap \Lambda_s \text{ and } \bigcap_p \Lambda^{(p)} = \Lambda,\end{equation}
we see that (\ref{dihedral intersect}) holds. Further 
\[\max \left \{d \left(\Lambda_1^{(p)} \right), d \left(\Lambda_2^{(p)} \right), d\left(\Lambda_3^{(p)} \right) \right \} = d\left(\Lambda^{(p)} \right)\]
and so $d(\Lambda)$ is the least common multiple of $d(\Lambda_1), d(\Lambda_2)$ and $d(\Lambda_3)$. \\

It remains to prove (\ref{D4 pairwise intersect}). Put
\[g_1 = \gcd(t_1 t_2 + t_3 t_4, t_1^2 + t_3^2, t_2^2 + t_4^2, t),\]
\[g_2 = \gcd(t_1 t_2 - t_3 t_4, t_1^2 - t_3^2, t_2^2 - t_4^2, t)\]
and
\[g_3 = \gcd(2t_2 t_3, 2t_2 t_4, 2t_1 t_3, t).\]

We shall show that $\gcd(g_1, g_2)$ is $1$ or $2$ and that

\begin{equation} \label{gcd equality} \gcd(g_1, g_2) = \gcd(g_1, g_3) = \gcd(g_2, g_3). \end{equation}

Notice that if $p$ divides $g_1$ then $t_1^2 \equiv - t_3^2 \pmod{p}$ and $t_2^2 \equiv - t_4^2 \pmod{p}$ while if $p$ divides $g_2$ then $t_1^2 \equiv t_3^2 \pmod{p}$ and $t_2^2 \equiv t_4^2 \pmod{p}$ and if $p$ divides $g_3$ then $p$ divides $2 t_2 t_3, 2 t_2 t_4$ and $2 t_1 t_3$. Thus if $p$ divides $\gcd(g_1, g_2)$ then $p$ divides $2t_1^2, 2t_2^2, 2 t_3^2$ and $2 t_4^2$; whence $p = 2$ since $\gcd(t_1, t_2, t_3, t_4) = 1$. Next suppose that $p$ divides $\gcd(g_1, g_3)$. Then $p$ divides $2 t_2 t_4$ and $t_2^2 \equiv - t_4^2 \pmod{p}$ and $p$ divides $2 t_1 t_3$ and $t_1^2 \equiv - t_3^2 \pmod{p}$. Since $\gcd(t_1, t_2, t_3, t_4) = 1$ we find that $p = 2$. Finally if $p$ divides $\gcd(g_2, g_3)$ then, as in the previous case, $p = 2$. Observe that 
\begin{equation} \label{2-order 0} 0 = \ord_2 g_1 = \ord_2 g_2 \leq \ord_2 g_3 \end{equation}
unless $(t_1, t_2, t_3, t_4)$ is congruent to $(1,0,1,0), (0,1,0,1)$ or $(1,1,1,1)$ modulo $2$ and in these cases 
\begin{equation} \label{2-order 1} 1 = \ord_2 g_1 = \ord_2 g_3 \leq \ord_2 g_2. \end{equation}
Thus (\ref{gcd equality}) follows from (\ref{2-order 0}) and (\ref{2-order 1}). \\

For each prime $p$ put $h_i = \ord_p g_i$ for $i = 1,2,3$. Then, by (\ref{gcd equality}), for some rearrangement $(i,j, k)$ of $(1,2,3)$ we have
\[h_{i} \geq h_{j} = h_{k}.\]
As in the proof of Lemma \ref{gen det char}, $\Lambda_i^{(p)}$ is defined by a single congruence modulo $p^{h - h_i}$ for $i = 1,2,3$. We check that $t$ divides the determinant of any matrix whose rows are taken from the rows determined by the coefficients of the congruence relations (\ref{D4 cong 1}), (\ref{D4 cong 2}), (\ref{D4 cong 3}), (\ref{D4 cong 4}), (\ref{D4 cong 5}), and (\ref{D4 cong 6}).  Furthermore $2t$ divides the determinant of such a matrix if $(t_1, t_2, t_3, t_4)$ is congruent to $(1,0,1,0), (0,1,0,1)$ or $(1,1,1,1)$ modulo $2$.
Since $h_{j} = h_{k}$ we see that the congruences modulo $p^{h - h_{j}}$ define identical lattices $\Lambda_{j}^{(p)}$ and $\Lambda_{k}^{(p)}$. Further, since $h_{i} \geq h_{j}$, $\Lambda_{j}^{(p)}$ is a sublattice of $\Lambda_{i}^{(p)}$ and (\ref{dihedral intersect}) follows when $\Aut F$ is equivalent to $\DD_4$. \\

Suppose now that $\Aut F$ is equivalent to $\DD_3$ under conjugation in $\GL_2(\bQ)$. There exists an element $T$ in $\GL_2(\bQ)$, as in (\ref{transfer matrix}), with $t_1, t_2, t_3$ and $t_4$ coprime integers for which 
\[\Aut F = T^{-1} \DD_3 T.\]
Define $t = t_1 t_4 - t_2 t_3$. The lattices $\Lambda_1, \Lambda_2, \Lambda_3$ and $\Lambda_4$ may be taken to be the lattices of integer pairs $(u,v)$ for which 
\[T^{-1} A_i T \binom{u}{v} \in \bZ^2\]
where 
\[A_1 = \begin{pmatrix} 0 & 1 \\ 1 & 0 \end{pmatrix}, A_2 = \begin{pmatrix} 1 & 0 \\ -1 & -1 \end{pmatrix}, A_3 = \begin{pmatrix} -1 & -1 \\ 0 & 1 \end{pmatrix} \text{ and } A_4 = \begin{pmatrix} 0 & 1 \\ -1 & -1 \end{pmatrix}.\] 
Thus $\Lambda_1$ consists of integer pairs $(u,v)$ for which 
\begin{equation} \label{D3 cong 1} (t_1 t_2 - t_3 t_4)u + (t_2^2 - t_4^2)v \equiv 0 \pmod{t} \end{equation}

and
\begin{equation} \label{D3 cong 2} (t_1^2 - t_3^2)u + (t_1 t_2 - t_3 t_4)v \equiv 0 \pmod{t}. \end{equation}
$\Lambda_2$ consists of integer pairs $(u,v)$ for which
\begin{equation} \label{D3 cong 3}  (t_1 t_2 + t_2 t_3 + t_1 t_4)u + (t_2^2 + 2 t_2 t_4)v \equiv 0 \pmod{t}
\end{equation}

and
\begin{equation} \label{D3 cong 4} (t_1^2 + 2 t_1 t_3)u + (t_1 t_2 + t_2 t_3 + t_1 t_4)v  \equiv 0 \pmod{t}.
\end{equation}
$\Lambda_3$ consists of integer pairs $(u,v)$ for which
\begin{equation} \label{D3 cong 5} (t_1 t_4 + t_2 t_3 + t_3 t_4)u + (2 t_2 t_4 + t_4^2) v \equiv 0 \pmod{t}
\end{equation}

and
\begin{equation} \label{D3 cong 6} (2 t_1 t_3 + t_3^2) u + (t_1 t_4 + t_2 t_3 + t_3 t_4)v \equiv 0 \pmod{t}. \end{equation}
$\Lambda_4$ consists of integer pairs $(u,v)$ for which
\begin{equation} \label{D3 cong 7} (t_1 t_2 + t_2 t_3 + t_3 t_4)u + (t_2^2 + t_2 t_4 + t_4^2)v \equiv 0 \pmod{t} \end{equation}
 
and
\begin{equation} \label{D3 cong 8} (t_1^2 + t_1 t_3 + t_3^2)u + (t_1 t_2 + t_1 t_4 + t_3 t_4) v \equiv 0 \pmod{t}. \end{equation}

For each prime $p$ dividing $t$ we put $h = \ord_p t$. Define $\Lambda_i^{(p)}$ for $i = 1,2,3,4$ to be the lattice of $(u,v)$ in $\bZ^2$ for which the congruences (\ref{D3 cong 1}) and (\ref{D3 cong 2}), (\ref{D3 cong 3}) and (\ref{D3 cong 4}), (\ref{D3 cong 5}) and (\ref{D3 cong 6}), and (\ref{D3 cong 7}) and (\ref{D3 cong 8}) respectively hold with $t$ replaced with $p^h$ and define $\Lambda^{(p)}$ to be the lattice for which all the congruences hold. We shall prove that for some reordering $(i, j, k, l)$ of $(1,2,3,4)$ we have
\begin{equation} \label{D3 pairwise intersect} \Lambda_{i}^{(p)} \supset \Lambda_{j}^{(p)} = \Lambda_{k}^{(p)} = \Lambda_{l}^{(p)}.\end{equation}
It then follows that 
\begin{equation} \label{D3 total intersect} \Lambda_r^{(p)} \cap \Lambda_s^{(p)} = \Lambda_1^{(p)} \cap \Lambda_2^{(p)} \cap \Lambda_3^{(p)} \cap \Lambda_4^{(p)} = \Lambda^{(p)} \end{equation}
for any pair $\{r,s\}$ from $\{1,2,3,4\}$. But since
\begin{equation} \label{D3 p-adic intersect} \bigcap_p \left(\Lambda_r^{(p)} \cap \Lambda_s^{(p)} \right) = \Lambda_r \cap \Lambda_s \text{ and } \bigcap_p \Lambda^{(p)} = \Lambda, \end{equation}
we conclude that (\ref{dihedral intersect}) holds. Further 
\[\max \left \{d \left(\Lambda_1^{(p)} \right), d \left(\Lambda_2^{(p)} \right), d\left(\Lambda_3^{(p)} \right), d\left(\Lambda_4^{(p)}\right) \right \} = d\left(\Lambda^{(p)} \right)\]
and so $d(\Lambda)$ is the least common multiple of $d(\Lambda_1), d(\Lambda_2), d(\Lambda_3)$ and $d(\Lambda_4)$. \\

It remains to prove (\ref{D3 pairwise intersect}). Put
\[g_1 = \gcd(t_1 t_2 - t_3 t_4, t_1^2 - t_3^2, t_2^2 - t_4^2, t),\]
\[g_2 = \gcd(t_1 t_2 + t_2 t_3 + t_1 t_4, t_1^2 + 2 t_1 t_3, t_2^2 + 2 t_2 t_4, t),\]
\[g_3 = \gcd(t_1 t_4 + t_2 t_3 + t_3 t_4, 2 t_1 t_3 + t_3^2, 2 t_2 t_4 + t_4^2, t),\]
and
\[g_4 = \gcd(t_1 t_2 + t_2 t_3 + t_3 t_4, t_1^2 + t_1 t_3 + t_3^2, t_2^2 + t_2 t_4 + t_4^2, t).\]
Suppose that $p$ is a prime which divides $\gcd(g_1, g_2)$. If $p$ divides $t_1$ then since $p$ divides $t_1^2 - t_3^2$ we see that $p$ divides $t_3$. Similarly if $p$ divides $t_2$ then since $p$ divides $t_2^2 - t_4^2$ we see that $p$ divides $t_4$. Since $t_1, t_2, t_3$ and $t_4$ are coprime either $p$ does not divide $t_1$ or $p$ does not divide $t_2$. In the former case since $p$ divides $t_1^2 + 2 t_1 t_3$ we find that $p$ divides $t_1 + 2 t_3$. Thus $t_1^2 \equiv 4 t_3^2 \pmod{p}$ and since $t_1^2 \equiv t_3^2 \pmod{p}$ we conclude that $p = 3$. In the latter case since $p$ divides $t_2^2 + 2 t_2 t_4$ we again find that $p = 3$. In a similar fashion we prove that if $p$ is a prime which divides $\gcd(g_i, g_j)$ for any pair $\{i,j\}$ from $\{1,2,3,4\}$ then $p = 3$. \\

Denote by
$E$ the set consisting of the $4$-tuples  $(1,1,1,1)$,  $(-1,-1,-1,-1)$, $(1,-1,1,-1)$, $(-1,1,-1,1)$,
$ (1,0,1,0)$, $(-1,0,-1,0)$, $(0,1,0,1)$ and  $(0,-1,0,-1)$.
One may check that if $(t_1, t_2, t_3, t_4)$ is not congruent modulo $3$ to an element of $E$ then for some reordering $(i,j, k, l)$ of $(1,2,3,4)$ we have 
\begin{equation} \label{3-order 0} 0 = \ord_3 g_{i} = \ord_3 g_{j} = \ord_3 g_{k} \leq \ord_3 g_{l}.\end{equation}
 
If $(t_1, t_2, t_3, t_4)$ is congruent modulo $3$ to an element of $E$ then there is some reordering $(i, j, k, l)$ of $(1,2,3,4)$ such that

\begin{equation} \label{3-order 1} 1 = \ord_3 g_{i} = \ord_3 g_{j} = \ord_3 g_{k} \leq \ord_3 g_{l}. \end{equation}
To see this we make use of the fact that 
\begin{equation} \label{3-order ineq} \ord_3 g_1 \leq \ord_3 (t_1^2 - t_3^2), \text{ } \ord_3 g_2 \leq \ord_3(t_1^2 + 2 t_1 t_3), \end{equation}
\[ \ord_3 g_3 \leq \ord_3 (2 t_1 t_3 + t_3^2) \text{ and } \ord_3 g_4 \leq \ord_3 (t_1^2 + t_1 t_3 + t_3^2), \]
to deal with the first six cases. To handle the remaining two cases, so when $(t_1, t_2, t_3, t_4)$ is congruent modulo $3$ to $(0,1,0,1)$ or $(0,-1,0, -1)$, we appeal to (\ref{3-order ineq}) but with $t_1$ and $t_3$ replaced by $t_2$ and $t_4$ respectively. \\

It now follows from (\ref{3-order 0}) and (\ref{3-order 1}) that $\gcd(g_1, g_2)$ is $1$ or $3$ and 
\begin{equation} \label{D3 pairwise gcd} \gcd(g_1, g_2) = \gcd(g_1, g_3)= \gcd(g_1,g_4)= \gcd(g_2,g_3)= \gcd(g_2,g_4)=\gcd(g_3,g_4).\end{equation}
For each prime $p$ put $h_i = \ord_p g_i$ for $i = 1,2,3,4$. Then, by (\ref{D3 pairwise gcd}) for some reordering $(i, j, k, l)$ of $(1,2,3,4)$ we have
\[h_{i} \geq h_{j} = h_{k} = h_{l}.\]
As in the proof of Lemma \ref{gen det char}, $\Lambda_{i}^{(p)}$ is defined by a single congruence relation modulo $p^{h - h_{i}}$ and $\Lambda_{j}^{(p)}, \Lambda_{k}^{(p)}$ and $\Lambda_{l}^{(p)}$ are defined by single congruences modulo $p^{h - h_j}$. We check that $t$ divides the determinant of any matrix whose rows are taken from the rows determined by the coefficients of the congruence relations (\ref{D3 cong 1}), (\ref{D3 cong 2}), (\ref{D3 cong 3}), (\ref{D3 cong 4}), (\ref{D3 cong 5}), (\ref{D3 cong 6}), (\ref{D3 cong 7}) and (\ref{D3 cong 8}) and that $3t$ divides the determinant of such a matrix if $(t_1, t_2, t_3, t_4)$ is congruent to an element of $E$.   Then since $h_j = h_k = h_l$ we see that the congruences modulo $p^{h - h_j}$ define identical lattices so 
\[\Lambda_{j}^{(p)} = \Lambda_{k}^{(p)} = \Lambda_{l}^{(p)}.\]
Further, since $h_i \geq h_j$, $\Lambda_{j}^{(p)}$ is a sublattice of $\Lambda_{i}^{(p)}$ and thus (\ref{D3 pairwise intersect}) holds and (\ref{dihedral intersect}) follows when $\Aut F$ is equivalent to $\DD_3$. \\

Finally we remark that (\ref{dihedral intersect}) holds when $\Aut F$ is equivalent to $\DD_6$ by the same analysis we used when $\Aut F$ is equivalent to $\DD_3$. 

\end{proof}

\begin{lemma} \label{binomial lemma} Let $a$ and $b$ be non-zero integers and let $d$ be an integer with $d \geq 3$. Put 
\[F(x,y) = ax^d + by^d.\]
If $a/b$ is not the $d$-th power of a rational number then when $d$ is odd
\[\Aut F = \left \{ \begin{pmatrix} 1 & 0 \\ 0 & 1 \end{pmatrix} \right \}\]
and when $d$ is even
\[\Aut F = \left \{ \begin{pmatrix} w_1 & 0 \\ 0 & w_2 \end{pmatrix} ; w_i \in \{1,-1\} , i=1,2 \right \} .\]
If $\dfrac{a}{b} = \dfrac{A^d}{B^d}$ with $A$ and $B$ coprime integers then when $d$ is odd
\[\Aut F = \left \{ \begin{pmatrix} 1 & 0 \\ 0 & 1 \end{pmatrix}, \begin{pmatrix} 0 & A/B \\ B/A & 0 \end{pmatrix} \right \} \]
and when $d$ is even
\[\Aut F = \left \{ \begin{pmatrix} w_1 & 0 \\ 0 & w_2 \end{pmatrix}, \begin{pmatrix} 0 & w_4 A/B \\  w_3B/A & 0 \end{pmatrix} ; w_i \in \{1,-1\},i=1,2,3,4 \right \}.\]
\end{lemma}

\begin{proof} Let 
\[U = \begin{pmatrix} u_1 & u_2 \\ u_3 & u_4 \end{pmatrix}\]
be an element of $\Aut F$. Then $u_1, u_2, u_3, u_4$ are rational numbers with
\begin{equation} \label{det pm} u_1 u_4 - u_2 u_3 = \pm 1. \end{equation}
Since $F(u_1 x + u_2 y, u_3 x + u_4 y) = F(x,y)$ we see on comparing coefficients that
\begin{equation} \label{diagonal sit} a u_1^d + b u_2^d = a, \text{ } au_3^d + bu_4^d = b  
\end{equation}
and
\begin{equation} \label{non-diagonal sit} a u_1^j u_2^{d-j} = -b u_3^j u_4^{d-j}  \end{equation}
for $j = 1, \cdots, d-1$. \\

Suppose that $u_1 u_2 \ne 0$. Then by (\ref{non-diagonal sit}), we have $u_3 u_4 \ne 0$ as well. Therefore we may write
\[\left(\frac{u_3}{u_1} \right) \left(\frac{u_4}{u_2}\right)^{d-1} = \left(\frac{u_3}{u_1} \right)^2 \left(\frac{u_4}{u_2}\right)^{d-2},\]
which implies that $u_1 u_4 - u_2 u_3 = 0$, contradicting (\ref{det pm}). Therefore, $u_1 u_2 = 0$ and similarly $u_3 u_4 = 0$. Further, by (\ref{det pm}), either $u_1 u_4 = \pm 1$ and $u_2 = u_3 = 0$ or $u_2 u_3 = \pm 1$ and $u_1 = u_4 = 0$. In the first case, by (\ref{diagonal sit}), we have $u_1^d = 1$ and $u_4^d = 1$, hence if $d$ is odd we have $u_1 = u_4 = 1$ while if $d$ is even we have $u_1 = \pm 1$ and $u_4 = \pm 1$. In the other case, by (\ref{diagonal sit}), we have $u_2^d = \dfrac{a}{b}$ and this is only possible if there exist coprime integers $A$ and $B$ with 
\[\frac{a}{b} = \frac{A^d}{B^d}.\]
In that case $u_2 = A/B$ if $d$ is odd and $u_2 = \pm A/B$ if $d$ is even. Thus, by (\ref{det pm}), $u_3 = B/A$ if $d$ is odd and $u_3 = \pm B/A$ if $d$ is even. Our result now follows. \end{proof}

\section{Proof of Theorems \ref{Rep by BF MT} and \ref{WF value theorem}}
\label{BF MT Proof}

If $\Aut F$ is conjugate to $\CC_1$ then every pair $(x,y) \in \bZ^2$ for which $F(x,y)$ is essentially represented with $ 0 < |F(x,y)| \leq Z$ gives rise to a distinct integer $h$ with $0 < |h| \leq Z$. It follows from Lemma \ref{EFZ lemma} and Lemma \ref{essential reps lemma} that
\[R_F(Z) = A_F Z^{\frac{2}{d}} + O_{F,\ep} \left(Z^{\beta_F + \ep}\right),\]
and we see that $W_F$ in this case is $1$. In a similar way we see that if $\Aut F$ is conjugate to $\CC_2$ then 
\[R_F(Z) = \frac{A_F}{2} Z^{\frac{2}{d}} + O_{F,\ep} \left(Z^{\beta_F + \ep}\right).\]

Next let us consider when $\Aut F$ is conjugate to $\CC_3$. Then for $A$ in $\Aut F$ with $A \ne I$ we have, by Lemma \ref{redundancy lemma}, $\Lambda(A) = \Lambda(A^2) = \Lambda$. Thus whenever $F(x,y) = h$ with $(x,y)$ in $\N_F^{(1)}(Z) \cap \Lambda$ there are two other elements $(x_1, y_1), (x_2, y_2)$ for which $F(x_i, y_i) = h$ for $i = 1,2$. When $(x,y)$ is in $\bZ^2$ but not in $\Lambda$ and $F(x,y)$ is essentially represented then $F(x,y)$ has only one representation. \\ 

Let $\omega_1, \omega_2$ be a basis for $\Lambda$ with $\omega_1 = (a_1, a_3)$ and $\omega_2 = (a_2, a_4)$. Put $F_\Lambda(x,y) = F(a_1 x + a_2 y, a_3 x + a_4 y)$ 
and notice that 
\begin{equation} \label{essential reps in lattice} |\N_F(Z) \cap \Lambda| = N_{F_\Lambda} (Z). \end{equation}
By Lemma \ref{Mahler Thm}
\begin{equation} \label{NFZ1} N_{F_\Lambda}(Z) = A_{F_\Lambda} Z^{\frac{2}{d}} + O_{F_\Lambda} \left(Z^{1/(d-1)}\right). \end{equation}
Since the quantity $\lvert \Delta(F) \rvert^{1/d(d-1)} A_F$ is invariant under $\GL_2(\bR)$
\begin{equation} \label{bean area invariance} \lvert \Delta(F) \rvert^{1/d(d-1)} A_F = \lvert \Delta(F_\Lambda) \rvert^{1/d(d-1)} A_{F_\Lambda} \end{equation}
and we see that
\begin{equation} \label{lattice jacobian} A_{F_\Lambda} = \frac{1}{d (\Lambda)} A_F = \frac{A_F}{m}. \end{equation} 
Therefore by (\ref{essential reps in lattice}), (\ref{NFZ1}) and  (\ref{lattice jacobian}) 
\begin{equation} \label{esslambda} |\N_F(Z) \cap \Lambda| = \frac{A_F}{m} Z^{\frac{2}{d}} + O_{F} \left(Z^{\frac{1}{d-1}}\right) . \end{equation}
Certainly $ \N_F^{(2)}(Z) \cap \Lambda $ is contained in $ \N_F^{(2)}(Z)$ and thus, by (\ref{esslambda}) and Lemma \ref{EFZ lemma},
\begin{equation} \label{esslambda1} |\N_F^{(1)}(Z) \cap \Lambda| = \frac{A_F}{m} Z^{\frac{2}{d}} + O_{F,\ep} \left(Z^{\beta_F + \ep}\right). \end{equation}
Each pair $(x,y)$ in $\N_F^{(1)}(Z) \cap \Lambda$ is associated with two other pairs which represent the same integer. Thus the pairs $(x,y)$ in $\N_F^{(1)}(Z) \cap \Lambda$ yield
\begin{equation} \label{extra C3} \frac{A_F}{3m} Z^{\frac{2}{d}} + O_{F,\ep} \left(Z^{\beta_F + \ep}\right) \end{equation} 
integers $h$ with $0<|h| \leq Z$.
By Lemma \ref{essential reps lemma} and (\ref{esslambda1}) the number of pairs $(x,y)$ in $\N_F^{(1)}(Z)$ which are not in $\Lambda$ is
\begin{equation} \label{not in lambda} \left(1 - \frac{1}{m}\right) A_F Z^{\frac{2}{d}} + O_{F,\ep} \left(Z^{\beta_F + \ep}\right)\end{equation} 
and each pair gives rise to an integer $h$ with $0 < |h| \leq Z$ which is uniquely represented by $F$. It follows from  (\ref{extra C3}),  (\ref{not in lambda}) and Lemma \ref{EFZ lemma} that when $\Aut F$ is equivalent to $\CC_3$ we have
\[R_F(Z) = \left(1 - \frac{2}{3m} \right) A_F Z^{\frac{2}{d}} + O_{F,\ep} \left(Z^{\beta_F + \ep}\right).\]
A similar analysis applies in the case when $\Aut F$ is equivalent to $\DD_1, \DD_2, \CC_4$ or $\CC_6$. These groups are cyclic with the exception of $\DD_2$ but $\DD_2/\{\pm I\}$ is cyclic and that is sufficient for our purposes. \\

We are left with the possibility that $\Aut F$ is conjugate to $\DD_3, \DD_4$ or $\DD_6$. We first consider the case when $\Aut F$ is equivalent to $\DD_4$. In this case, recall (\ref{esslambda1}), we have
\[|\N_F^{(1)}(Z) \cap \Lambda| = \frac{A_F}{m} Z^{\frac{2}{d}} + O_{F,\ep} \left(Z^{\beta_F + \ep}\right)\]
and since each $h$ for which $h = F(x,y)$ with $(x,y)$ in $\N_F^{(1)}(Z) \cap \Lambda$ is represented by $8$ elements of $\N_F^{(1)}(Z)$ the pairs $(x,y)$ of $\N_F^{(1)}(Z) \cap \Lambda$ yield
\begin{equation} \label{D4 total} \frac{A_F}{8m} Z^{\frac{2}{d}} + O_{F,\ep} \left(Z^{\beta_F + \ep}\right) \end{equation}
terms $h$ in $\R_F(Z)$. By Lemma \ref{redundancy lemma} we have $\Lambda_i \cap \Lambda_j = \Lambda$ for $1 \leq i < j \leq 3$; whence the terms $(x,y)$ in $\Lambda_1, \Lambda_2$ or  $\Lambda_3$ but not in $\Lambda$ for which $(x,y)$ is in $\N_F^{(1)}(Z)$ have cardinality
\[\left(\frac{1}{m_1} + \frac{1}{m_2} + \frac{1}{m_3} - \frac{3}{m}\right) A_F Z^{\frac{2}{d}} + O_{F,\ep} \left(Z^{\beta_F + \ep}\right).\]
If $(x,y)$ is in $\Lambda_1, \Lambda_2$ or  $\Lambda_3$ but not in $\Lambda$ and $h = F(x,y)$ is essentially represented then $h$ has precisely four representations. Accordingly the terms in
\[\N_F^{(1)}(Z) \cap \Lambda_i,    1 \leq i \leq 3\]
which are not in $\Lambda$ contribute 
\begin{equation} \label{D4 half total} \frac{1}{4} \left(\frac{1}{m_1} + \frac{1}{m_2} + \frac{1}{m_3} - \frac{3}{m} \right) A_F Z^{\frac{2}{d}} + O_{F,\ep} \left(Z^{\beta_F + \ep} \right) \end{equation} 
terms to $\R_F(Z)$. Finally the terms $(x,y)$ in $\N_F^{(1)}(Z)$ but not in $\Lambda_i$ for $i = 1,2,3$ have cardinality equal to
\[\left(1 - \frac{1}{m_1} - \frac{1}{m_2} - \frac{1}{m_3} + \frac{2}{m} \right) A_F Z^{\frac{2}{d}} + O_{F,\ep} \left(Z^{\beta_F + \ep}\right).\]
Each integer $h$ represented by such a term has $2$ representations and therefore these terms $(x,y)$ contribute
\begin{equation} \label{D4 quarter total} \frac{1}{2} \left(1 - \frac{1}{m_1} - \frac{1}{m_2} - \frac{1}{m_3} + \frac{2}{m} \right) A_F Z^{\frac{2}{d}} + O_{F,\ep} \left(Z^{\beta_F + \ep}\right) \end{equation}
terms to $\R_F(Z)$. It now follows from (\ref{D4 total}), (\ref{D4 half total}), (\ref{D4 quarter total}) and Lemma \ref{EFZ lemma} that
\[R_F(Z) = \frac{1}{2} \left(1 - \frac{1}{2m_1} - \frac{1}{2m_2} - \frac{1}{2m_3} + \frac{3}{4m} \right) A_F Z^{\frac{2}{d}} + O_{F,\ep}\left(Z^{\beta_F + \ep} \right),\]
as required. \\

We now treat the case when $\Aut F$ is conjugate to $\DD_3$. As before the pairs $(x,y)$ of $\N_F^{(1)}(Z) \cap \Lambda$ yield
\begin{equation} \label{D3 total} \frac{A_F}{6m} Z^{\frac{2}{d}} + O_{F,\ep} \left(Z^{\beta_F + \ep}\right) \end{equation}
terms in $\R_F(Z)$. Since $\Lambda_i \cap \Lambda_j = \Lambda$ for $1 \leq i < j \leq 3$ by Lemma \ref{redundancy lemma}, the pairs $(x,y)$ in $\N_F^{(1)}(Z) \cap \Lambda_i$ for $i = 1, 2, 3$ which are not in $\Lambda$ contribute 
\begin{equation} \label{D3 half total} \left(\frac{1}{2m_1} + \frac{1}{2m_2} + \frac{1}{2m_3} - \frac{3}{2m} \right) A_F Z^{\frac{2}{d}} + O_{F,\ep} \left(Z^{\beta_F + \ep}\right) \end{equation}
to $\R_F(Z)$. Further, the pairs $(x,y)$ in $\N_F^{(1)}(Z) \cap \Lambda_4$ which are not in $\Lambda$ contribute
\begin{equation} \label{D3 third total} \left(\frac{1}{3m_4} - \frac{1}{3m} \right) A_F Z^{\frac{2}{d}} + O_{F,\ep} \left(Z^{\beta_F + \ep}\right) \end{equation}
terms to $\R_F(Z)$. Furthermore the pairs $(x,y)$ in $\N_F^{(1)}(Z)$ which are not in $\Lambda_i$ for $i = 1,2,3,4$ contribute, by Lemma \ref{redundancy lemma}, 
\begin{equation} \label{D3 sixth total} \left(1 - \frac{1}{m_1} - \frac{1}{m_2} - \frac{1}{m_3} - \frac{1}{m_4} + \frac{3}{m} \right) A_F Z^{\frac{2}{d}} + O_{F,\ep} \left(Z^{\beta_F + \ep}\right) \end{equation}
terms to $\R_F(Z)$. It then follows from (\ref{D3 total}), (\ref{D3 half total}), (\ref{D3 third total}), (\ref{D3 sixth total}), and Lemma \ref{EFZ lemma} that
\[R_F(Z) = \left(1 - \frac{1}{2m_1} - \frac{1}{2m_2} - \frac{1}{2m_3} - \frac{2}{3m_4} + \frac{4}{3m} \right) A_F Z^{\frac{2}{d}} + O_{F,\ep} \left(Z^{\beta_F + \ep}\right) \]
as required. \\

When $\Aut F$ is equivalent to $\DD_6$ the analysis is the same as for $\DD_3$ taking into account the fact that $\Aut F$ contains $-I$ and so the weighting factor $W_F$ is one half of what it is when $\Aut F$ is equivalent to $\DD_3$. This completes the proof of Theorems \ref{Rep by BF MT} and \ref{WF value theorem}.  

\section{Proof of Corollary \ref{binomial} }
\label{binom and biquad}
 We first determine $W_F$. By Lemma \ref{binomial lemma}, if $a/b$ is not the $d$-th power of a rational then when $d$ is odd $\Aut F$ is equivalent to $\CC_1$ and, by Theorem \ref{WF value theorem}, $W_F = 1$ while when $d$ is even we have $m=1$ and $\Aut F$ is conjugate to $\DD_2$ and so by Theorem \ref{WF value theorem} we have $W_F = \dfrac{1}{4}$. Suppose that
\[\frac{a}{b} = \frac{A^d}{B^d} \]
with $A$ and $B$ coprime non-zero integers. If $d$ is odd then $\Aut F$ is equivalent to $\DD_1$ by Lemma \ref{binomial lemma}. Notice that
\[\begin{pmatrix} 0 & A/B \\ B/A & 0 \end{pmatrix} = \frac{1}{AB} \begin{pmatrix} 0 & A^2 \\ B^2 & 0 \end{pmatrix}\]
and that $A^2$ and $B^2$ are coprime integers. Therefore by Lemma \ref{gen det char} we have $m = |AB|$ and $W_F = 1 - \frac{1}{2|AB|}$ when $d$ is odd. If $d$ is even $\Aut F$ is equivalent to $\DD_4$ with $m_1 = 1, m_2 = m_3 = m = |AB|$ and by Theorem \ref{WF value theorem} we have
\[W_F = \frac{1}{4} \left(1 - \frac{1}{2|AB|}\right).\]

We now determine $A_F$. We first consider the case $F(x,y) = ax^{2k} + by^{2k}$, with $a$ and $b$ positive. Then
\[A_F = \iint_{ax^{2k} + by^{2k} \leq 1} dx dy.\]
Note that $A_F$ is four times the area of the region with ${ax^{2k} + by^{2k} \leq 1}$ and with $x$ and $y$ non-negative. Make the substitution $ax^{2k} = u, by^{2k} = uv, u,v \geq 0$. Then we see that
\begin{align*} \frac{1}{4} A_F & = \int_0^\infty \int_0^{\frac{1}{v+1}} \frac{1}{4k^2 (ab)^{1/2k} } u^{\frac{1}{k} - 1} v^{\frac{1}{2k} - 1} du dv \\
& = \frac{1}{4k (ab)^{1/2k}} \int_0^\infty \frac{v^{1/2k - 1}}{(1+v)^{1/k}} dv \\
\end{align*}
The above integral is $B(1/2k,1/2k)$ where $B(z,w)$ denotes the Beta function and thus, see $6.2.1$ of \cite{Dav},
\begin{align*}A_F = \frac{1}{k(ab)^{1/2k}} \frac{\Gamma^2(1/2k)}{\Gamma(1/k)}.
\end{align*}

Next, we treat the case $F(x,y) = ax^{2k} - by^{2k}$ with $a$ and $b$ positive. The region $\{(x,y) \in \bR^2 : |F(x,y)| \leq 1\}$ has equal area in each quadrant, so it suffices to estimate the area assuming $x, y \geq 0$. We further divide the region into two, depending on whether $ax^{2k} - by^{2k} \geq 0$ or not. Let $A_F^{(1)}$ denote the area of the region satisfying $x,y \geq 0, 0 \leq F(x,y) \leq 1$. We make the substitutions $ax^{2k} = u, by^{2k} = uv$ with $u,v \geq 0$. Then
\begin{align*}  \frac{1}{8}A_F=A_F^{(1)} & = \iint_{\substack{0 \leq ax^{2k} - by^{2k} \leq 1 \\ x,y \geq 0}} dx dy\\
& = \int_0^1 \int_0^{\frac{1}{1-v}} \frac{1}{4k^2 (ab)^{1/2k}} u^{\frac{1}{k} - 1} v^{\frac{1}{2k} - 1} du dv \\
& = \frac{1}{4k (ab)^{1/2k}} \int_0^1 \frac{v^{1/2k - 1}}{(1 - v)^{1/k}} dv \\
& = \frac{1}{4k (ab)^{1/2k}} \frac{\Gamma(1/2k) \Gamma(1 - 1/k)}{\Gamma(1-1/2k)}. 
\end{align*}

Next, we treat the case when $F(x,y) = ax^{2k+1} + by^{2k+1}$. We put $ax^{2k+1} = u$ and $by^{2k+1} = uv$. We thus obtain
\begin{align*} \frac{A_F}{2} |ab|^{1/(2k+1)} & = \frac{1}{2(2k+1)} \int_{-\infty}^\infty \frac{|v|^{\frac{1}{2k+1} - 1} dv}{|1 + v|^{2/(2k+1)}} \\
& = \frac{1}{2(2k+1)} \left( \int_{0}^\infty \frac{v^{-2k/(2k+1)} dv}{(1 + v)^{2/(2k+1)}} + \int_{0}^1 \frac{v^{-2k/(2k+1)}dv}{(1 - v)^{2/(2k+1)}} + \int_1^\infty \frac{v^{-2k/(2k+1)}dv}{(1-v)^{2/(2k+1)}}\right) \\
& = \dfrac{1}{2(2k+1)} \left(\dfrac{\Gamma^2\left(\frac{1}{2k+1}\right)}{\Gamma\left(\frac{2}{2k+1} \right)} +  \dfrac{\Gamma\left(\frac{1}{2k+1} \right) \Gamma\left(\frac{2k-1}{2k+1} \right)}{\Gamma\left(\frac{2k}{2k+1}\right)} + \frac{\Gamma\left(\frac{2k-1}{2k+1} \right) \Gamma \left(\frac{1}{2k+1} \right)}{\Gamma \left( \frac{2k}{2k+1} \right)}\right).
\end{align*}

\end{document}